\documentclass{amsart}
\usepackage{amssymb,latexsym}
\usepackage{amsfonts}
\usepackage[mathscr]{eucal}
\usepackage[all]{xy}
\usepackage{enumerate}
\usepackage[colorlinks,linkcolor=black,anchorcolor=black,citecolor=black]{hyperref}
\usepackage{fancyhdr}

\newtheorem{thm}{Theorem}[section]
\newtheorem{yl}[thm]{Lemma}
\newtheorem{tl}[thm]{Corollary}
\newtheorem{mt}[thm]{Proposition}
\theoremstyle{remark}
\newtheorem{zj}{Remark}[section]
\theoremstyle{definition}
\newtheorem{dy}{Definition}[section]
\numberwithin{equation}{section}
\allowdisplaybreaks[4]

\newcommand{\im}{\sqrt{-1}}

\newcommand{\pr}[2]{\langle#1,#2\rangle}%Inner product

\begin{document}
\title{Uniqueness of constant scalar curvature Sasakian metrics }
\author{Xishen Jin}
\address{Xishen Jin\\Key Laboratory of Wu Wen-Tsun Mathematics\\ Chinese Academy of Sciences\\School of Mathematical Sciences\\
University of Science and Technology of China\\
Hefei, 230026, P.R. China\\} \email{jinxsh@mail.ustc.edu.cn}
%\author{JiaWei Liu}
%\address{Jiawei Liu\\Key Laboratory of Wu Wen-Tsun Mathematics\\ Chinese Academy of Sciences\\School of Mathematical Sciences\\
%University of Science and Technology of China\\
%Hefei, 230026, P.R. China\\} \email{liujw24@mail.ustc.edu.cn}
\author{Xi Zhang}
\address{Xi Zhang\\Key Laboratory of Wu Wen-Tsun Mathematics\\ Chinese Academy of Sciences\\School of Mathematical Sciences\\
University of Science and Technology of China\\
Hefei, 230026,P.R. China\\ } \email{mathzx@ustc.edu.cn}
\thanks{The authors were supported in part by NSF in
China No.11131007, 11571332 and the Hundred Talents Program of CAS}
\begin{abstract}
  In this paper, we prove that the transverse Mabuchi $K$-energy functional is convex along the weak geodesic in the space of Sasakian metrics. As an application, we obtain the uniqueness of constant scalar curvature Sasakian metrics modulo automorphisms  for the transverse holomorphic structure.
\end{abstract}
\maketitle
%\newpage
%\tableofcontents
%\newpage
\section{Introduction}
Let $(M , g)$ be a connected oriented $2m+1$-dimensional Riemannian manifold. If the cone manifold $(C(M) , \tilde{g} ) =( M\times R^{+} , r^{2}g + dr^{2})$ is K\"ahler, we say $(M ,g)$ is a Sasakian manifold. It is well known that $(M,g)$ is a Sasakian-Einstein manifold  if and only if the K\"ahler cone $(C(M), \tilde{g} )$ is a Calabi-Yau cone.  Sasakian geometry  was introduced by Sasaki \cite{19} fifty years ago, and it can be seen as an odd-dimensional counterpart of K\"ahler geometry. Since Sasakian geometry has been proved to be a rich source for the production of positive Einstein metrics \cite{boyer2005einstein, GMSW2}, and the existence  of Sasakian-Einstein metrics is of great interest in the physics of the famous  Ads/CFT duality conjecture \cite{KW, Ma, martelli2006toric, martelli2008sasaki}, there has been renewed interest in Sasakian manifolds recently.

The aim of this paper is to study the uniqueness of Sasakian metrics with constant scalar curvature on compact Sasakian manifolds. The uniqueness of Sasakian-Einstein metrics was proved by Cho, Futaki and Ono (\cite{CFO}) for the toric case, and by Nitta and Sekiya (\cite{NS}) for the general case. In \cite{guan2012regularity}, Guan and the second author  studied the geodesic equation in the space of Sasakian metrics on a compact Sasakian manifold $(M,g)$, and obtained  the weak $C^2$ regularity of such geodesic.  Then, they proved that the constant scalar curvature Sasakian metric (cscS metric) is unique in each basic K\"ahler class if the first basic Chern class is either strictly negative or zero.

In \cite{berman2014convexity}, Berman and Berndtsson proved the convexity of the Mabuchi $K$-energy along the weak geodesic in the space of K\"ahler metrics by using Chen's weak $C^2$ regularity (\cite{blocki2009gradient,blocki2012,chen2000space}). As an application, they obtained the uniqueness of  constant scalar curvature K\"ahler metric (cscK metric) modulo automorphisms  on a compact K\"ahler manifold in a fixed K\"ahler class.

 A Sasakian manifold $(M,g)$ has one dimensional foliation $\mathcal{F}_\xi$ associated to the characteristic Reeb vector field $\xi$, which admits a transverse holomorphic structure.  Another Sasakian metric $g'$ is compatible with $g$ means that they have the same Reeb vector field, the same transverse holomorphic structure and the same holomorphic structure on the cone $C(M)$, see section 2 for details. In this paper, by using the weak $C^2$ regularity in \cite{guan2012regularity} and following the argument of Berman and Berndtsson (\cite{berman2014convexity}),  we prove the uniqueness of constant scalar curvature Sasakian metrics (cscS) up to the action of the identity component of the automorphism group for the transverse holomorphic structure, which is denoted to be $G_0$ in Definition \ref{automorphism}. In fact, we obtain the following theorem.

\medskip
%
%\begin{thm}
%\label{mainthm}
%  Let $(M, g)$ be a compact Sasakian manifold,  $g_{1}$ and $g_{2}$ are two constant scalar curvature Sasakian metrics compatible with $g$. Then there exists $\iota$ in the group $G_0$, such that $d\eta_{2}=\iota^* d\eta_{1}$.
%\end{thm}

\begin{thm}
\label{mainthm}
  Let $(M, g)$ be a compact Sasakian manifold,  $(\xi,\eta_1,\Phi_1,g_{1})$ and $(\xi,\eta_2,\Phi_2,g_{2})$ are two constant scalar curvature Sasakian metrics compatible with $g$. Then there exists $\iota$ in the group $G_0$, such that $d\eta_{2}=\iota^* d\eta_{1} =d\eta_1+ d_Bd_B^c\varphi_\iota$. Furthermore, for the two Sasakian structures, we have the following relations
  \begin{equation}
  \begin{split}
    \eta_2&=\eta_1+ d_B^c\varphi_\iota,\\
    \Phi_2&=\Phi_1-\xi\otimes d_B^c\varphi_\iota \circ \Phi,\\
    g_2&=\frac{1}{2}d\eta_2\circ (Id\otimes \Phi_2) +\eta_2\otimes \eta_2.
  \end{split}
  \end{equation}
\end{thm}

\medskip

This paper is organized as follow. In Section \ref{section2}, we will recall some preliminary results about Sasakian geometry, in particular, the weak geodesic established in \cite{guan2012regularity}. In Section \ref{section3}, we prove the convexity of the transverse Mabuchi $K$-energy $\mathcal{M}$ along the weak geodesic. In Section \ref{section4}, we give a proof of Theorem \ref{mainthm}, as an application of the convexity of $\mathcal{M}$.

\section{Preliminary Results in Sasakian Geometry}
\label{section2}

There are many structures on Sasakian manifold. A Sasakian manifold $(M,g)$ has a contact structure $(\xi,\eta,\Phi)$, and it also has a one-dimensional foliation $\mathcal{F}_\xi$, called the Reeb foliation, which has a transverse K\"ahler structure. Here, the killing vector field $\xi$ is called the characteristic or Reeb vector field, $\eta$ is called the contact $1$-form, and $\Phi$ is an $(1,1)$-tensor field which defines a complex structure  on the contact sub-bundle $\mathcal{D}=\operatorname{ker}\eta$. In the following, a Sasakian manifold will be denoted by $(M,\xi,\eta,\Phi,g)$, and the quadruple $(\xi,\eta,\Phi,g)$ will be called by a Sasakian structure on a manifold $M$.

Let $(M,\xi,\eta,\Phi,g)$ be a $(2n+1)$-dimensional Sasakian manifold, and let $\mathcal{F}_\xi$ be the characteristic foliation generated by $\xi$. A transverse holomorphic structure on $\mathcal{F}_\xi$ is given by an open covering $\{U_i\}_{i\in A}$ of $M$ and local submersion of $f_i:U_i\to \mathbb{C}^m$ with fibers of dimension $1$, such that for $i,j\in A$, there is a holomorphic isomorphism $\theta_{ij}$ of open sets of $\mathbb{C}^m$ such that $f_i=\theta_{ij}\circ f_j$ on $U_i\cup U_j$.

Now, we consider the quotient bundle of the foliation $\mathcal{F}_\xi$, $\nu(\mathcal{F}_\xi)=TM/L\xi$. The metric $g$ gives a bundle isomorphism $\sigma$ between $\nu(\mathcal{F}_\xi)$ and the contact sub-bundle $\mathcal{D}$, where $\sigma:\nu(\mathcal{F}_\xi)\to \mathcal{D}$ is defined by
$$
\sigma([X])=X-\eta(X)\xi.
$$
By this isomorphism, $\Phi|_\mathcal{D}$ induces a complex structure $\bar{J}$ on $\nu ( \mathcal{F}_\xi )$. $(\mathcal{D},\Phi|_\mathcal{D},d\eta)$ gives $M$ a transverse K\"ahler structure with transverse K\"ahler form $\frac{1}{2}d\eta$ and metric $g^T$ defined by $g^T=\frac{1}{2}d\eta(\cdot,\Phi\cdot)$. For the transverse metric $g^T$, one can define the transverse Levi-Civita connection $\nabla^T$ on $\mathcal{D}$ by
\[
\nabla^T_XY=
\begin{cases}
(\nabla_XY)^p, & X\in \mathcal{D}\\
[\xi,Y]^p, & X=\xi
\end{cases}
\]
where $Y$ is a section of $\mathcal{D}$ and $X^p$ the projection of $X$ onto $\mathcal{D}$. One can check that the transverse Levi-Civita connection is torsion-free and metric compatible. The transverse curvature relating to the above transverse connection will be denoted by $R^T(V,W)Z$, where $V,W\in TM$ and $Z\in \mathcal{D}$. We define the transverse Ricci curvature by
$$
\operatorname{Ric}^T(X,Y)=\langle R^T(X,e_i)e_i,Y\rangle_g,
$$
where $e_i$ is the orthonormal basis of $\mathcal{D}$ and $X,Y\in \mathcal{D}$. Furthermore, $\rho^T=\operatorname{Ric}^T(\Phi\cdot,\cdot)$ is called the transverse Ricci form analog to the K\"{a}hler geometry. The transverse scalar curvature $S^T$ is defined to be the trace of $\rho^T$ with respect to $g^T$. According to standard computation, we have that $S^{T}=S+2n$, where $S$ is the scalar curvature of $g$ in the usual sense.

We say a $p$-form $\theta$ is basic, if it satisfies that
$$
i_\xi\theta=0,\text{ and }L_\xi\theta=0.
$$
It is easy to check that $d\theta$ is also basic, if $\theta$ is, i.e. the exterior differential preserves basic forms. Let $\wedge^p_B(M)$ be the sheaf of germs of basic $p$-forms and $\Omega_B^p(M)=\Gamma(M,\wedge_B^p(M))$ be the set of all sections of $\wedge^p_B(M)$. The basic cohomology can be defined in a usual way(\cite{kacimi1990operateurs}). Let $\mathcal{D}^\mathbb{C}$ be the complexification of the sub-bundle $\mathcal{D}$, and we can decompose it into its eigenspaces with respect to $\Phi|_\mathcal{D}$, that is
$$
\mathcal{D}^\mathbb{C}=\mathcal{D}^{1,0}\oplus\mathcal{D}^{0,1}.
$$
Similarly, we have a splitting of the complexification of the bundle $\wedge^p_B(M)$ of basic $p$-forms on $M$,
$$
\wedge^p_B(M)\otimes\mathbb{C}=\mathop{\oplus}_{i+j=p}\mathop{\wedge}_B^{i,j}(M),
$$
where $\mathop{\wedge}\limits_B\limits^{i,j}(M)$ denotes the bundle of basic forms of type $(i,j)$.
Define $\partial_B$ and $\bar{\partial}_B$ by
\begin{equation*}
  \begin{split}
    \partial_B:\mathop{\wedge}\limits_B\limits^{i,j}(M) \to & \mathop{\wedge}\limits_B\limits^{i,j+1} (M),\\
    \bar{\partial}_B:\mathop{\wedge}\limits_B\limits^{i,j}(M) \to & \mathop{\wedge}\limits_B\limits^{i+1,j} (M),
  \end{split}
\end{equation*}
which is the decomposition of $d$. Let $d_B^c=\frac{1}{2} \sqrt{-1}(\bar{\partial}_B- \partial_B)$, and $d_B =d|_{\wedge^p_B}$. We have $d_B=\partial_B+ \bar{\partial}_B$ and $d_Bd_B^c =\im\partial_B \bar{\partial}_B$, and $d_B^2= (d_B^c)^2 =0$. The basic cohomology groups $H_B^{i,j}(M,\mathcal{F}_\xi)$ are fundamental invariants of a Sasakian structure  which enjoy many of the same properties as the Dolbeault cohomology of a K\"ahler structure. On Sasakian manifolds, the $\partial\bar{\partial}$-lemma holds for basic forms.

\medskip

\begin{mt}[\cite{kacimi1990operateurs}]
  Let $\theta$ and $\theta'$ be two real closed basic forms of type $(1,1)$ on a compact Sasakian manifold $(M,\xi,\eta,\Phi,g)$. If $[\theta]_B=[\theta']_B\in H_B^{1,1}(M,\mathcal{F}_\xi)$, then there is a real-valued basic function $\varphi$ such that
  \begin{equation*}
    \theta=\theta'+\sqrt{-1}\partial_B\bar{\partial}_B\varphi.
  \end{equation*}
\end{mt}

\medskip

Now we consider the deformation of the Sasakian structures. Let us denote the space of all smooth basic functions $\varphi$, i.e. $\xi\varphi=0$ on $(M,\xi,\eta,\Phi,g)$ by $C_B^\infty(M,\xi)$. And specially $$\mathcal{H}(\xi,\eta,\Phi,g)=\{\varphi \in C_B^\infty (M,\xi): \eta_\varphi\wedge (d \eta_\varphi)^n\neq 0\},$$ where $\eta_\varphi =\eta+d^c_B\varphi$, and $d\eta_\varphi=d\eta+d_Bd^c_B\varphi$. The space $\mathcal{H}(\xi,\eta,\Phi,g)$ is contractible, and we will denote it by $\mathcal{H}$ for simplicity. For $\varphi\in\mathcal{H}$, $(\xi,\eta_\varphi,\Phi_\varphi,g_\varphi)$ is also a Sasakian structure on $M$, where
$$
\Phi_\varphi=\Phi-\xi\otimes(d^c_B\varphi)\circ\Phi, \text{ and }g_\varphi=\frac{1}{2} d \eta_\varphi\circ(\mathop{Id}\otimes \Phi_\varphi) +\eta_\varphi\otimes \eta_\varphi.
$$
As in \cite{GuanZhang,guan2012regularity}, we define a functional $\mathcal{I}:\mathcal{H}\to \mathbb{R}$ by
\begin{equation}
\mathcal I (\varphi )=\sum_{p=0}^{n}\frac{n!}{(p+1)!
(n-p)!}\int_{M}\varphi (d\eta )^{n-p}\wedge
(\sqrt{-1}\partial_{B}\bar{\partial }_{B}\varphi )^{p} \wedge \eta .
\end{equation}

Set
\begin{equation*}
  \mathcal{H}_0=\{\varphi\in \mathcal{H}|\mathcal I (\varphi ) =0\},
\end{equation*}
and
\begin{equation*}
  \mathcal{K}=\{\text{transverse K\"ahler form in the basic }(1,1) \text{ class }[d\eta]_B\}.
\end{equation*}
Then we have $\mathcal{H}_0\cong \mathcal{K}$. In \cite{GuanZhang}, Guan and the second author proved that $\mathcal{H}_0$ is totally geodesic and totally convex in $\mathcal{H}$. And the tangent space of $\mathcal{H}_0$ at $\varphi$ is the set
 \begin{equation}
   \left.\mathrm{T}\mathcal{H}_0\right|_{\varphi}=\{\psi\in \mathcal{H}|\int_M \psi (d\eta_\varphi)^n\wedge \eta=0\}.
 \end{equation}

It is well known that both $(\xi,\eta_\varphi, \Phi_\varphi,g_\varphi)$ and $(\xi,\eta,\Phi,g)$ have the same transverse holomorphic structure on $\nu(\mathcal{F}_\xi)$ and the same holomorphic structure on the cone $C(M)$(\cite{boyer2006Sasakian, futaki2009transverse}). Obviously, these deformations of Sasakian structure deform the transverse K\"ahler form in the same basic $(1,1)$ class. As in \cite{boyer2008canonical}, we call this class the basic K\"ahler class of the Sasakian manifold $(M,\xi,\eta,\Phi,g)$. It should be noted that the contact bundle $\mathcal{D}$ may change under such deformations. We define $\mathcal{S}(\xi,\bar{J})$ to be the subset of all structures $(\tilde{\xi},\tilde{\eta}, \tilde{\Phi},\tilde{g})$ in $\mathcal{S}(\xi)$ such that the diagram
  \[
  \xymatrix{
TM \ar[d]_{\pi_\nu}\ar[r]^{\tilde{\Phi}}
            & TM  \ar[d]_{\pi_\nu}\\
\nu(\mathcal{F}_\xi)\ar[r]^{\bar{J}}
            & \nu(\mathcal{F}_\xi)}
  \]
commutes, where $\mathcal{S}(\xi)$ denotes all Sasakian structure $(\tilde{\xi},\tilde{\eta}, \tilde{\Phi},\tilde{g})$, such that $\tilde{\xi}=\xi$. The set $\mathcal{S}(\xi,\bar{J})$ consists of elements of $\mathcal{S}(\xi)$ with the same transverse holomorphic structure $\bar{J}$. For two different Sasakian structure in $\mathcal{S}(\xi,\bar{J})$, we have(\cite{boyer2006Sasakian, boyer2008canonical, Sparks2010Sasaki}):
\begin{yl}
             If $(\xi,\eta,\Phi,g)$ and $(\tilde{\xi},\tilde{\eta}, \tilde{\Phi},\tilde{g})$ are two Sasakian structures in $\mathcal{S}(\xi,\bar{J})$, then there exist real-valued basic functions $\varphi$ and $\psi$ and integral closed $1$-form $\alpha$, such that
  \begin{eqnarray*}
    \tilde{\eta} &=& \eta +d_{B}^c\varphi +d\psi +i(\alpha),\\
    \tilde{\Phi} &=& \Phi- \xi\otimes(\tilde{\eta}-\eta)\circ\Phi,\\
    \tilde{g} &=& \frac{1}{2} d\tilde{\eta}\circ(Id\otimes \tilde\Phi)+\tilde{\eta}\otimes\tilde{\eta},
  \end{eqnarray*}
  where $d\tilde{\eta}=d\eta+d_Bd^c_B\varphi$. In particular, if this two Sasakian structures induce the same holomorphic structure on the cone $C(M)$, we have that  $\tilde{\eta} = \eta +d_{B}^c\varphi $.
\end{yl}

\begin{dy}
\label{compatible}
  Given any Sasakian manifold $(M,\xi,\eta,\Phi, g)$, we say another Sasakian metric $g'$ is compatible with the Sasakian structure of  $(M,\xi,\eta,\Phi, g)$, if they have the same Reeb vector field,  the same transverse holomorphic structure $\nu(\mathcal{F}_\xi)$ and the same holomorphic structure on the cone $C(M)$.
\end{dy}

\medskip

By Lemma 2.2.,  if Sasakian structure $(\xi ,\eta ',\Phi ', g')$ is compatible with $(\xi,\eta,\Phi, g)$, then there must exists a basic function in $\varphi \in \mathcal{H}$ such that $\eta '=\eta +d_{B}^c\varphi $. In the following, we also say $\varphi $ the transverse K\"ahler potential of $g'$.  Similar to the K\"ahler case, we know that the average
$$
\bar{S}^T=\frac{\int_MS_\varphi^T(d\eta_\varphi)^n \wedge\eta_\varphi}{\int_M(d\eta_\varphi)^n \wedge\eta_\varphi} =\frac{\int_M2n\rho_\varphi^T\wedge(d\eta_\varphi)^{n-1} \wedge\eta_\varphi}{\int_M(d\eta_\varphi)^n \wedge\eta_\varphi}
$$
is a constant independent of the choice of $\varphi\in \mathcal{H}$. We say $g_\varphi \in \mathcal{K}$(or $\varphi\in \mathcal{H}$) is a constant scalar curvature Sasakian(cscS) metric, if $S_\varphi=\bar{S}^T+2n$, which is equivalent to $S_\varphi^T=\bar{S}^T$. Indeed, this equation is an elliptic equation of forth order. As in the K\"ahler case, the Mabuchi K-energy is a useful tool to consider such metrics. In the Sasakian case, the Mabuchi K-energy is defined in \cite{futaki2009transverse}, and we recall it in the following lemma.
\begin{yl}
  Let $\varphi_1$ and $\varphi_2$ are two basic functions in $\mathcal{H}$ and $\varphi_t\text{ }(t\in[a,b])$ be a path in $\mathcal{H}$ connecting $\varphi_1$ and $\varphi_2$. Then
  \begin{equation*}
  \begin{split}
  \mathcal{M}(\varphi_1,\varphi_2)&=-\int_a^b\int_M \dot{\varphi}_t (S_{\varphi_t}^T-\bar{S}^T) (d\eta_{\varphi_t})^n\wedge\eta_{\varphi_t}\wedge dt\\
  &=-\int_a^b\int_M \dot{\varphi}_t (S_{\varphi_t}^T-\bar{S}^T) (d\eta_{\varphi_t})^n\wedge\eta\wedge dt
  \end{split}
  \end{equation*}
  is independent of the path $\varphi_t$, where $\dot{\varphi}_t=\frac{d\varphi}{dt}$. Furthermore, $\mathcal{M}$ satisfies the $1$-cocycle condition
  \begin{equation}
    \mathcal{M}(\varphi_1, \varphi_2)+ \mathcal{M}( \varphi_2, \varphi_3)= \mathcal{M}(\varphi_1, \varphi_3)
  \end{equation}
  and
  \begin{equation}
  \label{Mabuchi1}
    \mathcal{M}(\varphi_1+C_1,\varphi_2+C_2)=\mathcal{M}(\varphi_1,\varphi_2)
  \end{equation}
  for any $C_1,C_2\in \mathbb{R}$.
\end{yl}

In view of \eqref{Mabuchi1}, $\mathcal{M}:\mathcal{H}\times \mathcal{H}\to \mathbb{R}$ factors through $\mathcal{H}_0\times\mathcal{H}_0$. Hence we can define the mapping $\mathcal{M}:\mathcal{K}\times\mathcal{K}\to \mathbb{R}$ by the identity $\mathcal{K}\cong\mathcal{H}_0$,
\begin{equation*}
\mathcal{M}(d\eta_{\varphi_1},d\eta_{\varphi_2}):=\mathcal{M}(\varphi_1,\varphi_2).
\end{equation*}

\begin{dy}
  The mapping $\mu:\mathcal{K}\to \mathbb{R}$, where $\mu(d\eta_\varphi)$ is defined by
  $$
  \mathcal{M}(d\eta_\varphi)=\mathcal{M}(d\eta,d\eta_\varphi)
  $$
  is called the K-energy map of the transverse K\"ahler class in $[d\eta]_B$. The mapping $\mathcal{M}:\mathcal{H}\to \mathbb{R}$, $\mathcal{M}(\varphi)=\mathcal{M}(0,\varphi)$ is also called by the K-energy map of the space $\mathcal{H}$.
\end{dy}

It is easy to see that cscS metric is a critical point of $\mathcal{M}$. In order to consider the uniqueness of such cscS metrics in the space $\mathcal{H}$, we will consider the convexity of the K-energy along the geodesics in $\mathcal{H}$.

In \cite{guan2012regularity}, Guan and the second author studied the geodesic equation in $\mathcal{H}$. Here, we recall some results. The Weil-Peterson metric in the space $\mathcal{H}$ is defined as
\begin{equation*}
  (\psi_1,\psi_2)_\varphi=\int_M\psi_1\psi_2(d\eta_\varphi)^n\wedge\eta, \text{ }\forall \psi_1,\psi_2\in T\mathcal{H}.
\end{equation*}
A nature connection of the metric can be defined to be
\begin{equation}
\label{connectionWP}
  D_{\dot{\varphi}}\psi=\dot{\psi}-\frac{1}{4}\pr{d_B\dot{\varphi}}{d_B\psi}_{g_\varphi}
\end{equation}
where $\psi\in C_B^\infty(M,\xi)=T\mathcal{H}$. And in particular, if $\varphi_t(t\in[a,b])$ is a geodesic, then it satisfies that
\begin{equation}
\label{geodesic1}
  \frac{\partial^2\varphi_t}{\partial t^2}-\frac{1}{4}|d_B\frac{\partial \varphi_t}{\partial t}|^2_{g_\varphi}=0.
\end{equation}

In \cite{guan2012regularity}, Guan and the second author reduced the geodesic equation \eqref{geodesic1} to the Dirichlet problem of a degenerate Monge-Amp\`{e}re type equation on the K\"ahler cone $C(M)=M\times \mathbb{R}^+$. And for convenience, we recall the key observation. They denoted a new function $\psi$ on $M\times [1,\frac{3}{2}] \subset C(M)$ by converting the time variable $t$ to the radial variable $r$ as follow,
\begin{equation}
  \psi(\cdot,r)= \varphi_{2(r-1)}(\cdot)+ 4\log r.
\end{equation}
By setting a $(1,1)$-form on $M \times [1,\frac{3}{2}]$ such that
\begin{equation}
  \Omega_\psi=\tilde\omega+\frac{r^2}{2}\sqrt{-1} (\partial\bar\partial\psi - \frac{\partial \psi}{\partial r} \partial\bar\partial r),
\end{equation}
where $\tilde{\omega}=\frac{1}{2}dd^cr^2$ is the fundamental form of the K\"ahler cone, they concluded that the geodesic equation \eqref{geodesic1} is equivalent to the following degenerate Monge-Amp\`{e}re type equation
\begin{equation}
\label{geo}
  (\Omega_\psi)^{n+1}=0, \text{ on }M\times[1,\frac{3}{2}].
\end{equation}

In this paper, we give another description of the geodesic equation similar to the K\"ahler case. And we will reduce it to a degenerate transverse Monge-Amp\`{e}re type equation on $M\times D$, where $D$ is an annulus in $\mathbb{C}$. We let $t=\log|\tau|$, and $\Psi(\cdot,\tau) =\varphi_t(\cdot)$, i.e. $\Psi$ is a radial function defined on $M\times D$, where $D=\{z\in\mathbb{C}| a\leq\log{z}\leq b\}$.

\begin{mt}
  The geodesic equation \eqref{geodesic1} or \eqref{geo} can be written as the following transverse Monge-Amp\`{e}re equation on $X=M\times D$ for the function $\Psi$,
  \begin{equation}
    (\pi^*d\eta+\tilde{d}\tilde{d}^c\Psi)^{n+1}\wedge\eta=0,
  \end{equation}
  where $\pi$ is the projection from $X$ to $M$ and $$\tilde{d}=d_B+d_\tau\text{ and }\tilde{d}^c=d_B^c+d_\tau^c, \text{ } d_\tau^C= \frac{\sqrt{1}}{2} (\bar\partial_\tau -\partial_\tau).$$
\end{mt}

\begin{proof}
  Direct computation implies that
  \begin{equation}
  \label{GEO}
    \begin{split}
     &(\pi^*d\eta+\tilde{d}\tilde{d}^c\Psi)^{n+1}\wedge\eta\\
%    =&(\pi^*d\eta+\sqrt{-1} \partial_B\bar\partial_B \Psi +\sqrt{-1} \partial_B\bar\partial_\tau \Psi +\sqrt{-1} \partial_\tau\bar\partial_B \Psi+ \sqrt{-1} \partial_\tau\bar\partial_\tau \Psi)^{n+1}\wedge\eta\\
%    =&(\pi^*d\eta_\varphi +\sqrt{-1} \partial_B\bar\partial_\tau \Psi +\sqrt{-1} \partial_\tau\bar\partial_B \Psi+ \sqrt{-1} \partial_\tau \bar\partial_\tau \Psi)^{n+1}\wedge\eta\\
%    =&(n+1)\sqrt{-1} \Psi_{\tau\bar\tau} \partial \tau \wedge \bar \partial \tau \wedge (\pi^*d\eta_\varphi)^n\wedge \eta +\\
%    &\frac{n(n+1)}{2}(\sqrt{-1} \partial_B \Psi_{\bar \tau}\wedge \bar \partial \tau ) \wedge (\sqrt{-1}\partial\tau\wedge \bar\partial_B \Psi_\tau)\wedge (\pi^*d\eta_\varphi)^{n-1}\wedge \eta+\\
%    &\frac{n(n+1)}{2}(\sqrt{-1}\partial \tau\wedge \bar\partial_B \Psi_{\tau} ) \wedge (\sqrt{-1} \partial_B \Psi_{\bar \tau}\wedge \bar\partial\tau)\wedge (\pi^*d\eta_\varphi)^{n-1}\wedge \eta\\
%    =&(n+1)\sqrt{-1} \Psi_{\tau\bar\tau} \partial \tau \wedge \bar \partial \tau \wedge (\pi^*d\eta_\varphi)^n\wedge \eta -\\
%    &(n+1)\left|\partial_B\Psi_{\bar{\tau}}\right|^2_{\pi^*d\eta_\varphi}\sqrt{-1} \partial \tau \wedge \bar \partial \tau \wedge (\pi^*d\eta_\varphi)^n\wedge \eta\\
    =&\frac{n+1}{4\left| \tau \right|^2}(\frac{\partial^2 \varphi}{\partial t^2}-\frac{1}{4}| d_B \frac{\partial \varphi}{\partial t} |^2_{g_\varphi})\sqrt{-1} \partial \tau \wedge \bar \partial \tau \wedge (\pi^*d\eta_\varphi)^n\wedge \eta.
    \end{split}
  \end{equation}

  According to the computation of \cite{guan2012regularity}, we know that
  \begin{equation}
    (\Omega_\psi)^{n+1}=(n+1)2^{-n}r^{2n+3} (\frac{\partial^2 \varphi}{\partial t^2}-\frac{1}{4}| d_B \frac{\partial \varphi}{\partial t} |^2_{g_\varphi}) dr\wedge \eta \wedge (d\eta_{\varphi})^n,
  \end{equation}
  and
  \begin{equation}
    \tilde{\omega}^{n+1}=(n+1)2^{-n}r^{2n+3} dr\wedge \eta \wedge (d\eta)^n.
  \end{equation}

  Furthermore, we have that
  \begin{equation}
  \frac{(\pi^*d\eta+\tilde{d}\tilde{d}^c\Psi)^{n+1}\wedge\eta}{\frac{n+1}{4\left| \tau \right|^2}\sqrt{-1} \partial \tau \wedge \bar \partial \tau \wedge (\pi^*d\eta)^n\wedge \eta}=\frac{(\Omega_\psi)^{n+1}}{\tilde{\omega}^{n+1}},
  \end{equation}
  which implies the result required.
\end{proof}

The equation \eqref{geodesic1} usually does not admit a smooth solution, i.e. we can not always find smooth geodesic in $\mathcal{H}$. But according to \cite{guan2012regularity} (Theorem 1.), there exists a unique weak $C^2$ solution $\Psi$, i.e. we have the following lemma.

\begin{yl}
\label{geodesic}
For any two functions $\varphi_a,\varphi_b\in \mathcal{H}$, there exists a unique geodesic path $\Psi$ connecting them, such that $\Psi|_{M\times {a}}=\varphi_a$, and $\Psi|_{M\times {b}}=\varphi_b$. In particular,
\begin{equation*}
||\Psi||_{C^1(M\times D)}+\mathop{\sup}_{M\times D} |\triangle \Psi|\leq C,
\end{equation*}
where $C$ is a constant depending only on $(\xi,\eta,\Phi,g)$, $||\varphi_a||_{C^{2,1}}$ and $||\varphi_b||_{C^{2,1}}$. Furthermore, $\Psi(\cdot,t)$ is basic and $\pi^*d\eta+\tilde{d}\tilde{d}^c\Psi\geq 0$ in the sense of $L^\infty$ on $M\times D$.
\end{yl}

\begin{zj}
  Indeed, we use the fact that the $C^1$-norm and $\triangle \Psi$ can be controlled by the function $\psi(\cdot,r)$ in \cite{guan2012regularity}, since $r$ is away from $0$.
\end{zj}

\section{Convexity of K-energy along weak geodesic}
\label{section3}
Since our original definition of the Mabuchi K-energy depends on the forth order derivative, we want to rewrite an explicit formula for it, which has an ``energy part" and ``entropy part" as in the K\"ahler case(see \cite{berman2014convexity} for the K\"ahler case). We begin with some notations. Given $\varphi\in \mathcal{H}$, we will write the energy as follow
$$
\mathcal{E}(\varphi)=\mathop{\sum}_{j=0}^{n} \int_M\varphi(d\eta_\varphi)^{n-j} \wedge(d\eta)^j \wedge \eta,
$$
and
$$
\mathcal{E}^T(\varphi)=\mathop{\sum}_{j=0}^{n-1} \int_M\varphi(d\eta_\varphi)^{n-j-1} \wedge(d\eta)^j\wedge T \wedge \eta,
$$
where $T$ is a basic form of $(1,1)$-type. It is easy to check that
\begin{equation}
\label{energy diff}
d\mathcal{E}|_\varphi=(n+1)(d\eta_\varphi)^n\wedge\eta,\text{ and }d\mathcal{E}^T|_\varphi =n(d\eta_\varphi)^{n-1}\wedge T\wedge\eta.
\end{equation}

\begin{zj}
  By writing $d\mathcal{E}|_\varphi=\Xi$, where $\Xi$ is a measure on $M$, we mean that
  \begin{equation*}
    \left. \frac{d\mathcal{E}(\varphi_t)}{dt} \right|_{t=0}= \int_M \dot{\varphi}_t|_{t=0} \Xi,
  \end{equation*}
  where $\{\varphi_t\}$ is a curve in $\mathcal{H}$ such that $\varphi_t|_{t=0}=\varphi$.
\end{zj}

Similarly, the second order variations of $\mathcal{E}$ and $\mathcal{E}^T$ are
\begin{eqnarray*}
d_\tau d^c_\tau\mathcal{E}(\varphi_t)= \int_M(\pi^*d\eta+\tilde{d} \tilde{d}^c\Psi)^{n+1}\wedge\eta, \\
d_\tau d^c_\tau\mathcal{E}^T (\varphi_t)=\int_M (\pi^*d\eta+\tilde{d} \tilde{d}^c\Psi)^n\wedge \pi^*T\wedge\eta.
\end{eqnarray*}
Finally, we consider the entropy of a measure $\mu$ relative to a reference measure $\mu_0$ is defined as follows, if $\mu$ is absolutely continuous with respect to $\mu_0$, then
\begin{equation*}
  H_{\mu_0}(\mu):=\int_M\log{\frac{\mu}{\mu_0}}\mu.
\end{equation*}
We have the following proposition for the entropy defined above(see \cite{berman2014convexity}).

\begin{mt}
\label{mtentropy}
  If $\mu_0$ and $\mu$ are probability measures on $M$ such that $\mu$ is absolutely continuous with respect to $\mu_0$, then
  \begin{equation*}
    H_{\mu_0}(\mu)=\sup_f(\int_Mf \mu-\log\int_Me^f\mu_0)
  \end{equation*}
  where the supermum is taken over all continuous functions on $M$. Furthermore, it is a convex function of the measure $\mu$ for the natural affine structure on the space of probability measures, and lower semicontinuous with respect to the weak *-topology.
\end{mt}

We can give the explicit formula of Mabuchi K-energy in terms of the energy and entropy defined above.

\begin{mt}
\label{mt3.2}
  With the notation $\mu_0=(d\eta)^n\wedge\eta$, the following formula holds for the Mabuchi K-energy on $\mathcal{H}$:
  \begin{equation*}
    \mathcal{M}(\varphi)=\frac{\bar{S}^T}{n+1}\mathcal{E}(\varphi) -2\mathcal{E}^{\rho^T_{d\eta}}(\varphi) +2H_{\mu_0}((d\eta_\varphi)^n\wedge\eta).
  \end{equation*}
\end{mt}

\begin{zj}
  This proposition can be proved by checking the first derivative of both side along a fixed path. And they are both equal to $0$, while $\varphi=0$. Furthermore, the new formula of Mabuchi K-energy is well-defined for any basic function with weak $C^2$-regularity, i.e. functions $\varphi$ such that $d\eta+d_Bd^c_B\varphi$ has $L^\infty$-coefficients.
\end{zj}

Now we consider the convexity of the Mabuchi K-energy along the weak geodesics, modifying the method of \cite{berman2014convexity}.

\begin{thm}
\label{weakconvexThm}
  Let $\varphi_t(\cdot)=\varphi_\tau(\cdot)=\Psi(\cdot,\tau)$ be a family of basic functions, such that $d\eta+d_Bd_B^c\varphi_\tau$ has $L^\infty$-coefficients, $\pi^*d\eta+\tilde{d} \tilde{d}^c\Psi\geq 0$, and $(\pi^*d\eta+\tilde{d} \tilde{d}^c\Psi)^{n+1}\wedge\eta=0$ on $M\times D$. Then the Mabuchi K-energy $\mathcal{M}(\varphi_\tau)$ is weakly subharmonic with respect to $\tau\in D$. In particular, $\mathcal{M}(\varphi_t)$ is weakly convex along the weak geodesic $\varphi_t$ connecting two given points in $\mathcal{H}$.
\end{thm}

\begin{proof}
  Similar to \cite{berman2014convexity}, we can prove that
  \begin{equation*}
    d_\tau d_\tau^c\mathcal{M}^{u}(\tau)=\int_M\tilde{d}\tilde{d}^c u \wedge(\pi^*d\eta+ \tilde{d}\tilde{d}^c \Psi)^n\wedge \eta,
  \end{equation*}
  where $u$ is a locally bounded function on $M\times D$ and
  \begin{equation*}
    \mathcal{M}^u(\tau):=\frac{\bar{S}^T}{n+1}\mathcal{E}(\varphi_\tau) -2\mathcal{E}^{\rho^T_{d\eta}}(\varphi_\tau) + 2\int_M u(\cdot,\tau)(d\eta_{\varphi_\tau})^n\wedge\eta.
  \end{equation*}

  We want to apply the above considerations to $u=\log{\frac{((d \eta +d_B d_B^c\Psi)^n \wedge \eta)}{(d\eta)^n \wedge\eta}}$, but this functions is not locally bounded. We will introduce a truncation in the following way as in \cite{berman2014convexity}. For a fixed positive number $A$, we define
  \begin{equation*}
    \Psi_A:=\max\{\log{\frac{(d\eta+d_Bd_B^c\Psi)^n}{(d\eta)^n\wedge \eta}}, \chi-A\}
  \end{equation*}
  where $\chi$ denotes to be $\chi=\pi^*\chi_0-\Psi$ for $\chi_0\in \mathcal{H}$, such that $d\eta_{\chi_0}=d\eta+d_Bd_B^c\chi_0\geq 0$. So we have that
  \begin{equation*}
    \tilde{d}\tilde{d}^c\chi=\pi^*d_Bd_B^c\chi_0-(\pi^*d\eta+\tilde{d} \tilde{d}^c\Psi)+\pi^*d\eta \geq -(\pi^*d\eta+\tilde{d} \tilde{d}^c\Psi).
  \end{equation*}

  We claim that for each $A>0$, $T_A=\int_M\tilde{d}\tilde{d}^c\Psi_A\wedge(\pi^*d\eta+ \tilde{d}\tilde{d}^c\Psi)^n\wedge\eta$ is defined by a nonnegative measure on $D$, i.e. for any nonnegative compact supported smooth function $\varrho$,
  \begin{equation}
     T_A(\varrho)=\int_D\varrho T_A\geq 0.
   \end{equation}
   In order to prove this claim, we will consider the localization of Sasakian manifold.

  In \cite{godlinski2000locally}, it has been proved that every Sasakian metric can be locally generated by a real function of $2n$ variables, i.e. the Sasakian analog of the K\"ahler potential for the K\"ahler geometry. This local property has also been applied in \cite{zhang2010some} to prove the existence of $\alpha$-invariant in Sasakian case. Let $\zeta_j$ be a partition of unity subordinate to a covering of coordinate patches such that the supported set of $\zeta_j$ is a subset of $\Omega_{p_j}^{-1}([-\frac{R}{2},\frac{R}{2}]\times B_o(\frac{R}{2}))$ for a local coordinate chart $(U_{p_j},\Omega_{p_j})$ such that $\Omega_{p_j}(p_j)=0$, and for convenience we write
  \begin{equation*}
  U_{p_j}=\Omega_{p_j}^{-1}([-\frac{R}{2},\frac{R}{2}]\times B_o(\frac{R}{2})).
  \end{equation*}
  Furthermore, we choose $R$ is sufficiently small such that we can write
  \begin{equation}
    d\eta=d_Bd_B^c \rho_j
  \end{equation}
  on $U_{p_j}$, for some real-valued basic function $\rho_j$ on $M$. The compactness of $M$ implies that we can choose such $R>0$ for all $p\in M$. For such partition of unity, we write $T_A=\sum_jT_A^j$, where
  \begin{equation}
  \begin{split}
  T_A^j&=\int_M\zeta_j\tilde{d}\tilde{d}^c\Psi_A\wedge(\pi^*d\eta+ \tilde{d}\tilde{d}^c\Psi)^n\wedge\eta\\
  &=\int_{U_{p_j}}\zeta_j\tilde{d}\tilde{d}^c\Psi_A\wedge(\pi^*d\eta+ \tilde{d}\tilde{d}^c\Psi)^n\wedge\eta.
  \end{split}
  \end{equation}

  Since $\rho_j\circ\pi(\tau,)+\varphi_\tau$ defines a plurisubharmonic(psh) function on $B_o(\frac{R}{2})\subset \mathbb{C}^m$, we know that we can approximate the measure $(d\eta+d_Bd^c_B\Psi)^n$ by the Bergman measure $\beta_{j,k}=\beta_{k(\rho_j\circ\pi(\tau,) +\varphi_\tau)}$ for the Hilbert space of all holomorphic functions on $B_o(\frac{R}{2})$ with the weight $k(\rho_j\circ\pi(\tau,) +\varphi_\tau)$, according to \cite{berman2014convexity}. More precisely, we consider the following measure on $D$,
  \begin{equation}
    T_{A,k}^j=\int_{U_{p_j}}\zeta_j\tilde{d}\tilde{d}^c\Psi_{A,k}\wedge(\pi^*d\eta+ \tilde{d}\tilde{d}^c\Psi)^n\wedge\eta,
  \end{equation}
  where $\Psi_{A,k}=\max\{\log{\beta_{j,k}}, \chi-A\}$. By the results on plurisubharmonic variation of Bergman kernels in \cite{berndtsson2006subharmonicity}, we know that
  \begin{equation}
    \tilde{d}\tilde{d}^c\log\beta_{j,k}\geq -k\tilde{d}\tilde{d}^c(\rho\circ\pi+\Psi) = -k (\pi^*d\eta+\tilde{d}\tilde{d}^c\Psi),
  \end{equation}
  on $B_o(\frac{R}{2})\times D$. So for $k\geq 1$, we have
  \begin{equation}
    \begin{split}
      T_{A,k}^j&=\int_{U_{p_j}}\zeta_j\tilde{d}\tilde{d}^c\Psi_{A,k}\wedge(\pi^*d\eta+ \tilde{d}\tilde{d}^c\Psi)^n\wedge\eta\\
      &=\int_{\Omega_{p_j}^{-1}([-\frac{R}{2},\frac{R}{2}]\times B_o(\frac{R}{2}))}\zeta_j\tilde{d}\tilde{d}^c\Psi_{A,k}\wedge(\pi^*d\eta+ \tilde{d}\tilde{d}^c\Psi)^n\wedge\eta\\
      &=\int_{[-\frac{R}{2},\frac{R}{2}]\times B_o(\frac{R}{2})}(\zeta_j\circ\Omega_{p_j}^{-1}) \Omega_{p_j}^{-1*}(\tilde{d}\tilde{d}^c\Psi_{A,k}\wedge(\pi^*d\eta+ \tilde{d}\tilde{d}^c\Psi)^n\wedge\eta)\\
      &=\int_{-\frac{R}{2}}^{\frac{R}{2}}\int_{B_o(\frac{R}{2})}(\zeta_j \circ \Omega_{p_j}^{-1}) \Omega_{p_j}^{-1*}(\tilde{d}\tilde{d}^c\Psi_{A,k} \wedge(\pi^*d\eta+ \tilde{d}\tilde{d}^c\Psi)^n)\wedge dx\\
      &\geq-\int_{-\frac{R}{2}}^{\frac{R}{2}} \int_{B_o(\frac{R}{2})}k(\zeta_j\circ\Omega_{p_j}^{-1}) \Omega_{p_j}^{-1*}((\pi^*d\eta+ \tilde{d}\tilde{d}^c\Psi)^{n+1})\wedge dx\\
      &=-k\int_{U_{p_j}}\zeta_j(\pi^*d\eta+ \tilde{d} \tilde{d}^c\Psi)^{n+1} \wedge \eta\\
      &=0,
    \end{split}
  \end{equation}
  where we use the geodesic equation in the last equality and the property $$dd^c\max\{u,v\}\geq \max\{dd^cu,dd^cv\}$$ as a current for two psh functions $u$ and $v$ in the inequality.

  Hence, invoking the dominated convergence theorem gives the following local weak convergence
  \begin{equation}
    \lim_{k\to \infty} T_{A,k}^j=T_A^j.
  \end{equation}
  In particular, $T_A^j\geq 0$, so is $T_A$ which concludes the proof of the theorem.
\end{proof}

Similar to the argument in \cite{berman2014convexity}, we can also prove that $\mathcal{M}(\varphi)$ is continuous along the weak geodesics, and hence is convex.

\begin{thm}
\label{convex}
  $\mathcal{M}$ is continuous and convex along the weak geodesics given in Lemma \ref{geodesic}.
\end{thm}

\begin{proof}
  Let $\{\zeta_j\}$ be the partition of unity as in the proof of Theorem \ref{weakconvexThm}, and $\kappa_\varepsilon(s)$ be a sequence of strictly convex functions with $\kappa'_\varepsilon \geq 1$ on $\{s|s\leq C\}$ tending to $s$ as $\varepsilon \to 0$, where $C$ is the upper bound of $f_{A,k}(\tau)=\log\Psi_{A,k}\frac{(d\eta_{\varphi_\tau})^n\wedge \eta}{(d \eta )^n \wedge \eta}$, for $\Psi_{A,k}$ and $\Psi$ in Theorem \ref{weakconvexThm}. In particular, the functions $\kappa_\varepsilon (s)$ can be defined by
  \begin{equation}
    \kappa_\varepsilon(s)=-\sqrt{(s-C-2\sqrt{\varepsilon})^2-\varepsilon}+C.
  \end{equation}

  With the notations in the proof of Theorem \ref{weakconvexThm}, we denote
  \begin{eqnarray}
     H^j_{\varepsilon, A,k}(\tau)&=&\int_M\zeta_j \kappa_\varepsilon (f_{A,k}(\tau)) (d\eta)^n\wedge \eta ,\\
     \mathcal{E}_j(\varphi)&=&\sum\limits_{i=0}^{n} \int_M\zeta_j\varphi(d\eta_\varphi)^{n-i} \wedge(d\eta)^i \wedge \eta,\\
     \mathcal{E}_j^T(\varphi)&=&\sum\limits_{i=0}^{n-1} \int_M\zeta_j\varphi(d\eta_\varphi)^{n-i-1} \wedge(d\eta)^i\wedge T \wedge \eta.
  \end{eqnarray}
  Furthermore, let $\mathcal{M}_{\varepsilon}^{\Psi_{A,k}} = \mathop{\sum}\limits_j \mathcal{M}_{\varepsilon,j}^{\Psi_{A,k}}$, where
  \begin{equation}
  \mathcal{M}_{\varepsilon,j}^{\Psi_{A,k}} =\frac{\bar{S}^T}{n+1} \mathcal{E}_j- 2\mathcal{E}^{\rho_{d\eta}^T}_j+ 2 H^j_{\varepsilon,A,k}.
  \end{equation}
  Similar to the proof of Theorem \ref{weakconvexThm}, we consider $d_\tau d_\tau^c\mathcal{M}_{\varepsilon,j}^{\Psi_{A,k}}$. Indeed, with the notation $T_{\varepsilon,A,k}^j= \int_M \zeta_j\kappa'_\varepsilon \tilde{d}\tilde{d}^c \Psi_{A,k} \wedge (\pi^*d \eta+ \tilde{d} \tilde{d}^c\Psi)^n \wedge \eta \geq 0$, we have that
  \begin{equation}
  \begin{split}
    d_\tau d_\tau^c \mathcal{M}_{\varepsilon,j}^{\Psi_{A,k}}=&-2\int_M \zeta_j(1-\kappa'_\varepsilon)(\pi^*d\eta+\tilde{d} \tilde{d}^c\Psi)^n\wedge \pi^*\rho_{d\eta}^T\wedge\eta\\
    \mathrel{\phantom{=}} &+2\int_M \zeta_j \kappa''_\varepsilon d_\tau f\wedge d_\tau^c f\wedge(d\eta)^n\wedge \eta+T_{\varepsilon,A,k}^j\\
    \geq &-C_0
  \end{split}
  \end{equation}
  where we use the convexity of $\kappa_\varepsilon$ and the fact that  $\kappa'_\varepsilon \in [1 ,\frac{2}{\sqrt{3}}]$.We know that $\mathcal{M}_{\varepsilon,j}^{\Psi_{A,k}}+ C_0t^2$ is convex, since it is weak convex and the local Bergman kernels depend continuously on $\tau$. Let $k$ tend to $\infty$, we know that $\mathcal{M}_{\varepsilon,j}^{\Psi_A}+C_0t^2$ is convex, where $\mathcal{M}_{\varepsilon,j}^{\Psi_A}$ is the functional replacing $\Psi_{A,k}$ by $\Psi_{A}$ in $\mathcal{M}_{\varepsilon,j}^{\Psi_{A,k}}$.

  If we sum over $j$, we know that $\mathcal{M}_{\varepsilon}^{\Psi_A}+\tilde{C}t^2=\mathop{\sum}\limits_{j} \mathcal{M}_{\varepsilon,j}^{\Psi_A}+\tilde{C}t^2$ is also convex. We conclude that $\mathcal{M}^{\Psi_A}+\tilde{C}t^2$ is also convex by $\varepsilon \to 0$. So $\mathcal{M}^{\Psi_A}$ is continuous on $(0,1)$ and upper semi-continuous on $[0,1]$. In particular, $\mathcal{M}^{\Psi_A}$ is convex, since we have known that it is weakly convex. Let $A\to \infty$, we know that $\mathcal{M}$ is convex, which means that $\mathcal{M}$ is continuous on $(0,1)$ and upper semi-continuous on $[0,1]$.

  In order to complete the proof of this theorem, we just prove that $\mathcal{M}$ is continuous. Indeed, $(d\eta_{\varphi_t})^n\wedge\eta$ is continuous in the weak $*$-topology, if $\varphi_t$ is the geodesic with weak $C^2$ regularity. Combining the continuity of $\mathcal{E}$ and $\mathcal{E}^T$ and the entropy part is lower semi-continuous in the weak $*$-topology of measure(see Proposition \ref{mtentropy}), we know that $\mathcal{M}$ is also continuous.\qedhere

\end{proof}

\begin{yl}
\label{diff mabuchi increasing}
  Given $u_0,u_1$ in $\mathcal{H}$, let $u_t$ be the corresponding weak geodesic. Then
  \begin{equation}
    \lim\limits_{t\to 0^+}\frac{\mathcal{M}(u_t)- \mathcal{M}(u_0)}{t}\geq \int_M \frac{du_t}{dt}_{t=0^+}(\bar{S}^T-S^T_{u_0})(d\eta_{u_0})^n\wedge\eta.
  \end{equation}
\end{yl}

\begin{proof}
  We first consider the entropy part $H_{\mu_0}$ of $\mathcal{M}$. According to the convexity of $H_{\mu_0}$ with respect to the affine structure of probability measures, we know that
  \begin{equation}
    H_{\mu_0}(\nu_1)-H_{\mu_0}(\nu_0)\geq \left. \frac{d H_{\mu_0}(\nu_s)}{ds} \right|_{s=0},
  \end{equation}
  where $\nu_s=s\nu_1 +(1-s) \nu_0$. The monotone convergence implies that $$\left. \frac{d H_{\mu_0} (\nu_s)}{ds} \right|_{s=0}= \int_M \log\frac{\nu_0}{\mu_0}(\nu_1- \nu_0).$$ In particular, $\nu_1=(d\eta_{u_t})^n\wedge \eta$ and $\nu_0=(d\eta_{u_0})^n\wedge \eta$, we have that
  \begin{equation}
  \begin{split}
    \mathrel{\phantom{=}}&\frac{1}{t}(H_{\mu_0}((d\eta_{u_t})^n\wedge \eta)-H_{\mu_0}((d\eta_{u_0})^n\wedge \eta)) \\
    \geq &\int_M \log\frac{(d\eta_{u_0})^n\wedge \eta}{\mu_0}\frac{1}{t} ((d\eta_{u_t})^n \wedge \eta- (d\eta_{u_0})^n\wedge \eta)\\
    =&\int_M \frac{u_t-u_0}{t}d_Bd_B^c(\log\frac{(d\eta_{u_0})^n\wedge \eta}{\mu_0}) \wedge ( \sum\limits_{j=0}^{n-1} (d\eta_{u_t})^{n-j-1} \wedge (d\eta_{u_0})^j) \wedge \eta.
  \end{split}
  \end{equation}
  We get the estimate of the entropy part by $t\to 0$, i.e.
  \begin{equation}
  \begin{split}
    \mathrel{\phantom{=}}&\lim\limits_{t\to 0^+}\frac{1}{t}(H_{\mu_0}((d\eta_{u_t})^n\wedge \eta)-H_{\mu_0}((d\eta_{u_0})^n\wedge \eta))\\
    \geq& n \int_M \left. \frac{du_t}{dt}\right|_{t=0^+}(\rho_{d\eta}^T - \rho_{d\eta_{u_0}}^T) \wedge (d\eta_{u_0})^n \wedge \eta.
  \end{split}
  \end{equation}

  From the equation \eqref{energy diff}, we know that
  \begin{equation}
    \lim\limits_{t\to 0^+}\frac{\bar{S}^T}{n+1}\frac{\mathcal{E}(u_t) - \mathcal{E}(u_0)}{t} =\bar{S}^T \int_M \left. \frac{du_t}{dt} \right|_{t=0^+}(d\eta_{u_0})^n\wedge\eta
  \end{equation}
  and
  \begin{equation}
    \lim\limits_{t\to 0^+}\frac{\mathcal{E}^{\rho_{d\eta}^T}(u_t) - \mathcal{E}^{\rho_{d\eta}^T}(u_0)}{t} =n\int_M\left. \frac{du_t}{dt} \right|_{t=0^+}(d\eta_{u_0})^{n-1}\wedge {\rho_{d\eta}^T} \wedge\eta.
  \end{equation}

  According to the Proposition \ref{mt3.2}, we get the required result.
\end{proof}

A direct consequence of Theorem \ref{convex} and the lemma above is the following corollary:

\begin{tl}
  If $u_0\in\mathcal{H}$ is a cscS metric, then $u_0$ is a minimum of $\mathcal{M}$ in $\mathcal{H}$.
\end{tl}

\begin{proof}
  If the result is not true, we assume that there exists $u_1\in \mathcal{H}$ satisfies that $\mathcal{M} (u_1) < \mathcal{M}(u_0)$. Let $\{u_t\}_{t=0}^1$ is the corresponding weak geodesic connecting $u_0$ and $u_1$. According to the convexity of $\mathcal{M}$ with respect to $t$, we have that
  \begin{equation}
    \mathcal{M}(u_{t})\leq t\mathcal{M}(u_1)+ (1-t) \mathcal{M}(u_0),
  \end{equation}
  which is equivalent to
  \begin{equation}
    \frac{\mathcal{M}(u_{t})- \mathcal{M}(u_{0})}{t}\leq \mathcal{M}(u_{1})- \mathcal{M}(u_{0}).
  \end{equation}
  Hence $\lim\limits_{t\to 0^+}\frac{\mathcal{M}(u_{t})- \mathcal{M}(u_{0})}{t}\leq \mathcal{M}(u_{1})- \mathcal{M}(u_{0}).$ However, according to $u_0$ is a cscS metric and Lemma \ref{diff mabuchi increasing}, we know that
  \begin{equation}
    \lim\limits_{t\to 0^+}\frac{\mathcal{M}(u_{t})- \mathcal{M}(u_{0})}{t} \geq 0,
  \end{equation}
  which is a contradiction with $\mathcal{M} (u_1) < \mathcal{M}(u_0)$.
\end{proof}

\section{Uniqueness of cscS metrics}
\label{section4}
In this section, We  consider the uniqueness of cscS metrics in the space $\mathcal{H}$ modulo the action generated by the  Hamiltonian transverse holomorphic vector fields.

\medskip

\begin{dy}[\cite{boyer2008canonical}]
Fixed a transverse holomorphic structure $(\nu (\mathcal{F}_\xi), \bar{J} )$ on the characteristic foliation $\mathcal{F}_\xi$. A complex vector field $X$ on $M$ is called a  transverse holomorphic vector field  if it satisfies:
\begin{enumerate}
    \item $\pi( [\xi , X]) =0$;
    \item $\bar{J}(\pi(X))=\sqrt{-1}\pi (X)$;
    \item $\pi ([Y, X])-\sqrt{-1}\bar{J}\pi ([Y, X])=0$, $\forall Y$ satisfying $\bar{J}\pi(Y)=-\sqrt{-1}\pi (Y)$,
\end{enumerate}
where $\pi $ is the the projection to $\nu (\mathcal{F}_\xi)$.
\end{dy}

\medskip

Let $h^T(\xi,\bar{J})$ denote the set of all transverse holomorphic vector fields. One can easily check that $h^T(\xi,\bar{J})$ is a Lie algebra. Let $X$ be a transverse holomorphic vector fields and $f$ be a real-valued function, then $X+f\xi$ is also a  transverse holomorphic vector fields. So, $h^T(\xi,\bar{J})$ cannot have finite dimension. But, by \cite{boyer2008canonical}, we know that $h^T(\xi,\bar{J})/L_\xi$ has finite dimension.

%Before proving this we recall the canonical Sasakian metrics following the work of \cite{boyer2008canonical}, \cite{guan2012regularity}.

\medskip

\begin{dy}[\cite{CFO}]
\label{automorphism}
Let $(M, \xi, \eta , \Phi , g)$ be a compact Sasakian manifold. The automorphism group $G$ of the transverse holomorphic structure is the group of all biholomorphic automorphisms of $(C(M), J)$ which commute with the holomorphic flow generated by $\xi -\sqrt{-1} J \xi$ . Its identity component will be denoted by $G_{0}$.
\end{dy}

\medskip

It is well known (\cite{CFO}) that $G_{0}$ acts on the space of all Sasakian metrics on $M$ which is compatible with $g$.

\medskip

\begin{dy}[\cite{futaki2009transverse}]
A complex vector field $X$ on $(M, \xi, \eta , \Phi , g)$ is called a Hamiltonian transverse holomorphic vector field if it is transverse holomorphic and the complex valued basic function $\psi_{X}=\sqrt{-1} \eta (X)$ satisfies:
\begin{eqnarray}
\bar{\partial }_{B}\psi_{X} =-\frac{\sqrt{-1}}{2}d\eta (X , \cdot ).
\end{eqnarray}
\end{dy}

\medskip

Let $ \verb"g" ^{T}(\xi,\bar{J})$ denote the set of all Hamiltonian transverse holomorphic vector fields, it is easy to see that $\verb"g"^T(\xi,\bar{J})$ is a Lie algebra.

\medskip

\begin{mt}[\cite{CFO}]
Let $(M, \xi, \eta , \Phi , g)$ be a compact Sasakian manifold. Then the Lie algebra of the automorphism group $G$ of transverse holomorphic structure is the Lie algebra $\verb"g"^{T} (\xi,\bar{J})$ of all Hamiltonian transverse holomorphic vector fields.
\end{mt}

\medskip

Furthermore, we can conclude that $\verb"g"^{T} (\xi,\bar{J})$ is of finite dimension since the dimension of $h^T(\xi,\bar{J})/L_\xi$ is finite. Indeed, if $X_1,X_2$ are two different Hamiltonian transverse holomorphic vector fields but $X_1=X_2+f\xi$ for some basic function $f$, then by the definition of Hamiltonian transverse holomorphic vector field, we have that
\begin{equation}
  \bar{\partial}_B(\sqrt{-1} \eta(X_2+f\xi))=-\frac{\sqrt{-1}}{2}d\eta(X_2+f\xi,\cdot),
\end{equation}
or equivalently,
\begin{equation}
  \bar{\partial}_B(\sqrt{-1} \eta(X_2)+\sqrt{-1}f)=-\frac{\sqrt{-1}}{2}d\eta(X_2,\cdot).
\end{equation}
Hence $\bar{\partial}_Bf=0$, i.e. $f\equiv \text{Constant}$, which implies that
\begin{equation}
  \verb"g"^{T} (\xi,\bar{J})\subset h^T(\xi,\bar{J})/L_\xi \oplus \mathbb{C}\xi,
\end{equation}
and $\verb"g"^{T} (\xi,\bar{J})$ is of finite dimension.

Let $\varphi $ be a complex valued  basic function, the Hamiltonian vector field $\partial_{d\eta}^\# \varphi$ of $\varphi$ corresponding to the transverse K\"ahler form $d\eta$ is defined by:
\begin{enumerate}
  \item $\bar{J}( \pi( \partial_{d\eta}^\#\varphi)) = \sqrt{-1} \pi( \partial_{d\eta}^\# \varphi)$,
  \item $\varphi=\sqrt{-1}\eta(\partial_{d\eta}^\#\varphi)$,
  \item $\bar{\partial}_B \varphi(\cdot) =-\frac{\sqrt{-1}}{2} d\eta (\partial_{d \eta}^\#\varphi , \cdot)$.
\end{enumerate}
As in \cite{boyer2008canonical}, we denote the Lichnerowicz operator $L_{d \eta}^B$ as follow:
\begin{equation}
  L_{d \eta}^B \varphi : = \frac{1}{4} (\triangle^2_B\varphi + 4(\rho ^T , \sqrt{-1} \partial_B \bar{\partial}_B \varphi) + 2 i_{\partial_{d \eta}^\# \varphi} \partial S^T ).
\end{equation}
According to \cite{boyer2008canonical,futaki2009transverse}, the kernel $\mathcal{H}_{d \eta}^B$ of $L_{d \eta}^B$ is just all the basic functions $\varphi\in \mathcal{H}$ such that $\partial_{d \eta}^\# \varphi$ is transverse holomorphic.

Now we prove the uniqueness of cscS metrics. It should be noted that if $\mathcal{H}^B_{d \eta}$ is trivial, i.e. there is no Hamiltonian transverse holomorphic vector fields over $M$, Guan and the second author proved that such metric is unique in \cite{guan2012regularity}. And we will use the method of perturbation in \cite{berman2014convexity} to prove the uniqueness of cscS metrics while $\mathcal{H}^B_{d \eta}$ is non-trivial.

Let $\mu>0$ be a basic smooth volume form on $M$ with the following normalization
\begin{equation}
  \int_M \mu=\int_M(d\eta)^n\wedge\eta.
\end{equation}
Similar to \cite{berman2014convexity}, we define the function
\begin{equation}
  \tilde{\mathcal{F}}_\mu(u)=\int_Mu \mu-\frac{\mathcal{E}(u)}{n+1} :=I_\mu(u)-\frac{\mathcal{E}(u)}{n+1}
\end{equation}
where $\mathcal{E}$ is the energy functional in section \ref{section3}. The differential of $\tilde{\mathcal{F}}_\mu$ at $u\in\mathcal{H}$ is
\begin{equation}
\label{differentialofF}
  \left.d\tilde{\mathcal{F}_\mu}\right|_{u} =\mu- (d\eta_u)^n\wedge \eta.
\end{equation}

For the functional $I_\mu$, we have the following inequality as in \cite{berman2014convexity}.

\begin{mt}
\label{geodist}
$I_\mu$ is strictly convex along weak $C^2$ geodesic $\{\varphi_t\}$, in the sense that if $f(t):=I_\mu(\varphi_t)$ is affine, then for any $t$, $d\eta_{\varphi_t}=d\eta_{\varphi_0}$. More precisely, if $d\eta_{\varphi_t}=d\eta+dd^c\varphi_t\leq Cd\eta$ and $\mu\geq A(d\eta)^n\wedge\eta$, then
\begin{equation}
\label{distance}
  f'(1)-f'(0)\geq \frac{\delta A}{C^{n+1}} d(d\eta_{\varphi_0},d\eta_{\varphi_1})^2,
\end{equation}
where $\delta>0$ only depends on $\mu$, $\eta$ and $M$, and $d(d\eta_{\varphi_0}, d\eta_{\varphi_1})$ is the distance between $d\eta_{\varphi_0}$ and $d\eta_{\varphi_1}$ defined in \cite{guan2012regularity}.
\end{mt}

\begin{proof}
  The $C^1$-regularity of $\varphi_t$ implies that $f'(t)=\int_{M}\dot{\varphi}_t\mu$ is continuous. In order to get the estimate \eqref{distance}, we will apply the result of \cite{guan2012regularity}, that is we can approximate $\{\varphi_t\}$, by a smooth sequence $\{\varphi_{t,\varepsilon}\}$, and $\ddot{\varphi}_{t,\varepsilon} - \left| \partial_B \dot{\varphi}_{t,\varepsilon} \right|^2_{d\eta_{\varphi_{t,\varepsilon}}}\geq 0$. Furthermore, we know that the constant $C$ is still valid, since $\varphi_{t,\varepsilon}$ and $\dot{\varphi}_{t,\varepsilon}$ converges to $\dot{\varphi}_t$.

  With the notation above, we can compute directly
  \begin{equation}
  \begin{split}
    \frac{d^2}{dt^2}I_{\mu}(\varphi_{t,\varepsilon})&=\int_{M} \ddot \varphi_{t,\varepsilon} \mu\\
    &\geq \int_M \left| \partial_B \dot{\varphi}_{t,\varepsilon} \right|^2_{d\eta_{\varphi_{t,\varepsilon}}} \mu\\
    &=C^{-1} \int_{M} |\partial_B \dot \varphi_{t,\varepsilon}|^2_{d\eta} \mu\\
    &\geq \frac{\delta}{C}\int_{M} \left| \dot \varphi_{t,\varepsilon} -a_{t,\varepsilon}  \right|^2 \mu,\\
  \end{split}
  \end{equation}
%  \begin{equation}
%  \begin{split}
%    \frac{d^2}{dt^2}f(\varphi_{t,\varepsilon})&=\int_{M} \ddot \varphi_{t,\varepsilon} \mu\\
%    &=\lim\limits_{\varepsilon\to 0}\int_{M} \ddot \varphi_t \frac{\mu}{(d \eta_{ \varphi_t})^n\wedge \eta +\varepsilon \mu} (d \eta_{ \varphi_t})^n\wedge \eta\\
%    &=\lim\limits_{\varepsilon\to 0}\int_{M} |d_B \dot \varphi_t|^2_{d\eta_{\varphi_t}} \frac{\mu}{(d \eta_{ \varphi_t})^n\wedge \eta +\varepsilon \mu} (d \eta_{ \varphi_t})^n\wedge \eta\\
%    &\geq C^{-1}\lim\limits_{\varepsilon\to 0}\int_{M} |d_B \dot \varphi_t|^2_{d\eta} \frac{\mu}{(d \eta_{ \varphi_t})^n\wedge \eta +\varepsilon \mu} (d \eta_{ \varphi_t})^n\wedge \eta\\
%    &=C^{-1} \int_{M} |d_B \dot \varphi_t|^2_{d\eta} \mu\\
%    &\geq \frac{\delta}{C}\int_{M} \left| \dot \varphi_t -a_t  \right|^2 \mu,\\
%  \end{split}
%  \end{equation}
  where $a_{t,\varepsilon}$ is the average of $\dot\varphi_{t,\varepsilon}$ under the measure $\mu$.

  Integrating from $0$ to $1$, we get that
  \begin{equation}
    \left. \frac{dI_{\mu}(\varphi_{t,\varepsilon})}{dt} \right|_{t=1}- \left. \frac{dI_{\mu}(\varphi_{t,\varepsilon})}{dt} \right|_{t=0} \geq \frac{\delta}{C}\int_0^1\int_M \left| \dot \varphi_{t,\varepsilon} -a_{t,\varepsilon}  \right|^2 \mu dt.
  \end{equation}
  The continuity of the differential of $I_{\mu}$ implies that
  \begin{equation}
  f'(1) - f'(0) \geq \frac{\delta}{C}\int_0^1\int_M \left| \dot \varphi_{t} -a_{t}  \right|^2 \mu dt,
  \end{equation}
  where $a_{t}$ is the average of $\dot\varphi_{t}$ under the measure $\mu$. Hence if $f$ is affine, we have $\dot \varphi_t=a_t$, i.e. $d\eta_{ \varphi_t} = d\eta_{\varphi_0}$.

  For the last statement, we argue as following
  \begin{equation}
    \begin{split}
    f'(1)-f'(0)&\geq \frac{\delta}{C} \int_{0}^1 \int_M \left| \dot \varphi_t -a_t \right|^2\mu  dt\\
    &\geq \frac{A\delta}{C} \int_{0}^1 \int_M \left| \dot \varphi_t -a_t \right|^2 (d\eta)^n\wedge \eta \wedge dt \\
    &\geq \frac{A\delta}{C^{n+1}} \int_{0}^1 \int_M \left| \dot \varphi_t -a_t \right|^2 (d\eta_{\varphi_t})^n\wedge \eta \wedge dt\\
    &\geq \frac{A\delta}{C^{n+1}} d(d\eta_{\varphi_0},d\eta_{\varphi_1})^2.\qedhere
    \end{split}
  \end{equation}
\end{proof}

By the result of \cite{kacimi1990operateurs}, we know that for any basic smooth volume form $\mu$ on $M$, there exists a basic function $u\in \mathcal{H}$ such that
\begin{equation}
  (d\eta+d_Bd^c_Bu)^n\wedge\eta=\mu.
\end{equation}
Hence we can get the following lemma for the functional $\tilde{\mathcal{F}}_\mu$:

\begin{yl}
\label{funcitonal}
  Let $\mu$ and $\nu$ be two smooth basic volume forms with total mass equal to $\int_M(d\eta)^n\wedge\eta$. Then for all $\varphi \in \mathcal{H}$:
  \begin{equation}
    | \tilde{\mathcal{F}}_\mu(\varphi) - \tilde{\mathcal{F}}_\nu(\varphi) | \leq C_{\mu,\nu}.
  \end{equation}
\end{yl}

\begin{proof}
  Assume $\mu=(d\eta_{u_\mu})^n\wedge \eta$ and $\nu=(d\eta_{u_\nu})^n\wedge \eta$. Then
  \begin{align*}
  %\begin{split}
    &\mathrel{\phantom{=}}\left| \tilde{\mathcal{F}}_\mu(\varphi) -\tilde{\mathcal{F}}_\nu(\varphi) \right|\\
    &=\left| I_{\mu}(\varphi )- I_{\nu} (\varphi)\right|\\
    &= \left| \int_M \varphi d_B d_B^c(u_{\mu} -u_{\nu}) (\mathop{\sum}_{i=0}^{n-1}(d_Bd_B^c u_\mu)^i\wedge (d_Bd_B^c u_\mu)^{n-i-1})\wedge \eta \right|\\
    &= \left| \int_M (u_{\mu} -u_{\nu}) d_B d_B^c \varphi (\mathop{\sum}_{i=0}^{n-1}(d_Bd_B^c u_\mu)^i\wedge (d_Bd_B^c u_\mu)^{n-i-1})\wedge \eta\right|\\
    &= \left| \int_M (u_{\mu} -u_{\nu}) (d\eta_\varphi- d\eta)\wedge (\mathop{\sum}_{i=0}^{n-1}(d_Bd_B^c u_\mu)^i\wedge (d_Bd_B^c u_\mu)^{n-i-1})\wedge \eta\right|\\
    &\leq \mathop{OSC}(u_{\mu} -u_{\nu}) =C_{\mu,\nu}.\qedhere
  %\end{split}
  \end{align*}
\end{proof}

We denote $F(u)=d\mathcal{M}|_u$ to be the differential of $\mathcal{M}$ at $u$. Here, we regard $F(u)$(or $F$ if it is clear) as a $1$-form on $\mathcal{H}$, i.e. for $v$ in the tangent space of $\mathcal{H}$ at $u$, $F(u)$ acts on $v$ by
\begin{equation}
F(u)\cdot v=\int_Mv(\bar{S}^T-S^T_u)(d\eta)^n\wedge\eta.
\end{equation}
According to the computation of \cite{guan2012regularity}, the Hessian of $\mathcal{M}$ at $u\in\mathcal{H}$ is equal to
\begin{equation}
\begin{split}
  (Hess \mathcal{M})_{u}(\psi_0,\psi_1)&=(D d\mathcal{M})(\psi_0,\psi_1)= (D F)(\psi_0,\psi_1)\\
  & =\frac{1}{2} \int_M\psi_0\mathfrak{D}^*_{d\eta_u} \mathfrak{D}_{d\eta_u}\psi_1(d\eta_u)^n\wedge \eta\\
  &=\frac{1}{2} \int_M\pr{\mathfrak{D}_{d\eta_u}\psi_0 } {\mathfrak{D}_{d\eta_u}\psi_1}_{d\eta_u}(d\eta_u)^n\wedge \eta,
\end{split}
\end{equation}
where $D$ is the connection defined in \eqref{connectionWP} and $\mathfrak{D}_{d\eta_u}=\bar\partial_B\partial_{d\eta_u}^\#$.

For a basic smooth volume form $\nu$ on $M$, we define a functional $G_\nu$ on $\mathcal{H}$ by
\begin{equation}
  G_\nu\cdot w=\int_Mw \nu.
\end{equation}
and consider the solution $v$ to the equation
\begin{equation}
\label{Dif Equ}
  D_vF|_u=G_\nu.
\end{equation}
Similarly to \cite{berman2014convexity}, we know that \eqref{Dif Equ} is solvable if and only if $G_\nu\cdot w=0$ for all $w\in \mathcal{H}^B_{d\eta_u}$, i.e. all basic $w$ such that $\partial_{d\eta_u}^\#w$ is transverse holomorphic.

Let $u_0,u_1$ be two different\footnote{Saying $u_0$ and $u_1$ are different, we mean $u_0-u_1$ is not transverse constant, i.e. $d\eta_{u_0}$ and $d\eta_{u_1}$ are two different transverse K\"ahler metrics.} smooth basic functions with cscS metrics such that $\mathcal{I}(u_0)=\mathcal{I}(u_1)=0$. Modifying the argument of \cite{berman2014convexity} for K\"ahler case, we show that after a preliminary modification of $u_i(i=0,1)$ by applying an action generated by $\verb"g"^{T} (\xi, \bar{J})$, the equation
\begin{equation}
  D_{v_i}F|_{u_i}=-G_{\nu_i}
\end{equation}
is solvable, where $G_{\nu_i}$ is the differential of $\tilde{\mathcal{F}}_\mu$ in \eqref{differentialofF} at $u_i$\footnote{We will also use the notation $u_i$ for the function after action.}. For the convenience of the readers, we will give details of the proof here.

\begin{yl}
\label{HamiltonianGeodesic}
  If $V$ is a Hamiltonian transverse holomorphic vector field, then it determines a geodesic ray in $\mathcal{K}$ following the flow of $V$.
\end{yl}

\begin{proof}
  We denote $d\eta_{\tilde{u}_t}=d \eta+d_B d_B^c u_t= \exp (tV)^* d \eta_{u_0}$, where $\tilde{u}_t \in \mathcal{H}$ is smooth with respect to $t$ and $\tilde{u}_0=u_0$. Since $V$ is a Hamiltonian transverse holomorphic vector field, we know that there exists $h_{\tilde{u}_t}^V \in \mathcal{H}$ such that $V=\partial^\#_{d\eta_{\tilde{u}_t}} h_{\tilde{u}_t}^V$ and $h_{\tilde{u}_t}^V=h_{\tilde{u}_0}^V+V(\tilde{u}_t-\tilde{u}_0)$. Hence we have that
  \begin{equation}
  \begin{split}
    L_Vd\eta_{\tilde{u}_t} &=\frac{d}{dt} \exp(tV)^* d\eta_{u_0}\\
                  &=\sqrt{-1}d_Bd_B^c\dot{\tilde{u}}_t \\
                  &=\sqrt{-1}d_Bd_B^c h_{\tilde{u}_t}^V,
  \end{split}
  \end{equation}
  which implies that $\dot{\tilde{u}}_t=h_{\tilde{u}_t}^V+f(t)$ for $f\in C^{\infty}(\mathbb{R})$. Let $\hat{u}_t=\tilde{u}_t - \int_0^t f dt$, then we can conclude that $\dot{\hat{u}}_t=h_{\tilde{u}_t}^V$ and $V=\partial^\#_{d\eta_{\tilde{u}_t}} \dot{\hat u}_t$. Furthermore,
  \begin{equation}
  \begin{split}
     &\frac{d}{dt}\int_M \dot{\hat u}_t (d\eta_{\hat u_t})^n \wedge \eta\\
    =&\frac{d}{dt}\int_Mh_{\tilde u_t}^V (d\eta_{\tilde u_t})^n \wedge \eta\\
    =&\int_M V(\dot{u}_t) (d\eta_{\tilde u_t})^n \wedge \eta +\int_M h_{\tilde u_t}^V \triangle_B \dot{\tilde u}_t (d\eta_{\tilde u_t})^n \wedge \eta \\
    =&\int_M V(\dot{\tilde u}_t) (d\eta_{\tilde u_t})^n \wedge \eta - \int_M \langle \partial^\#_{d\eta_{\tilde u_t}} h_{\tilde u_t}^V, \nabla_B \dot{u}_t \rangle_{d\eta_{u_t}} (d\eta_{\tilde u_t})^n \wedge \eta \\
    =&\int_M V(\dot{\tilde u}_t) (d\eta_{\tilde u_t})^n \wedge \eta - \int_M V(\dot{\tilde u}_t) (d\eta_{\tilde u_t})^n \wedge \eta \\
    =&0,
  \end{split}
  \end{equation}
  so we have that $\int_M \dot{\hat u}_t (d\eta_{\hat u_t})^n \wedge \eta=C$. If we denote $u_t=\hat u_t -C t$, then we can conclude that $\mathcal{I}(u_t)\equiv 0$, i.e. $u_t\in \mathcal{H}_0$, since $\mathcal{I}(C_t)=C_t$ where $C_t$ is some constant of $t$. So we have that
  \begin{equation}
    \begin{split}
      \ddot{u}_t=&\dot{h}_{\tilde u_t}^V= V(\dot{\tilde u}_t)= V(\dot{\hat u}_t)\\
      =&d\eta_{\tilde u_t} (\partial^\#_{d \eta_{\tilde u_t}} \dot{\hat u}_t, \nabla^B \dot{\hat u}_t) =\left|\left|\partial_B\dot{\hat u}_t\right|\right|_{d\eta_{\tilde u_t}} =\left|\left|\partial_B\dot{u}_t\right|\right|_{d\eta_{ u_t}}.
    \end{split}
  \end{equation}
  Hence $u_t$ is a geodesic, and so is $d\eta_{u_t}$, i.e. the ray determined by $V$.
\end{proof}

\begin{mt}
  Let $S$ be the submanifold of $\mathcal{H}_0$ consisting of all potentials of metrics $\iota^*d\eta_{u_i}(i=0,1)$, where $\iota$ ranges over the actions generated by $\verb"g"^{T} (\xi,\bar{J})$. Then $\tilde{\mathcal{F}}_\mu$ has a minimum and hence a critical point on $S$. This implies that $G_{\nu_i}$ annihilates all basic functions $w$ such that $\partial_{d\eta_{u_i}}^\#w$ is transverse holomorphic.
\end{mt}
\begin{proof}
   According to Lemma \ref{HamiltonianGeodesic}, any Hamiltonian transverse holomorphic vector field $V$ determines a geodesic ray starting at $u_i$. And $S$ is the union of all such rays. If $\mu=(d\eta_{u_i})^n\wedge\eta$, then $u_i$ is a critical point of $\tilde{ \mathcal{ F } }_\mu$. Since $\tilde{\mathcal{F}}_\mu$ is strictly convex along each ray, it follows that $\tilde{\mathcal{F}}_{\mu}$ is proper on each ray. And here we say a function $f(t)$ is proper if and only if \(\mathop{\lim}\limits_{t\to +\infty} f(t) =+\infty.\) Since the dimension of $\verb"g"^{T} (\xi,\bar{J})$ is finite, $\tilde{\mathcal{F}}_\mu$ is proper on $S$ in this case. Lemma \ref{funcitonal} implies that $\tilde{\mathcal{F}}_\mu$ is proper on $S$ for any choice of $\mu$. Hence it has a minimum on $S$. And for convenience, we still use the notation $u_i$ for the function which achieves the minimum of $\tilde{\mathcal{F}}_\mu$ on $S$.

   Let $\iota_t=\exp{(tV)}$ be the one-parameter group of transformations determined by a Hamiltonian transverse holomorphic vector field $V=\partial_{d\eta_{u_i}}^\#h_{u_i}^V$. For $u_{i,t}\in \mathcal{H}_0$ which is the potential of $d\eta_{u_{i,t}}= \iota_t^*d \eta_{u_i}$, we have
   \begin{equation}
     L_Vd\eta_{u_{i}}=\sqrt{-1}\partial_B\bar{\partial}_Bh_{u_i}^V,
   \end{equation}
   and
   \begin{equation}
     \frac{d(g_t^*d\eta_{u_i})}{dt} =\frac{d(d\eta+\sqrt{-1}\partial_B \bar{\partial}_B u_{i,t})}{dt}=\sqrt{-1}\partial_B\bar{\partial}_B\dot{u}_{i,t}.
   \end{equation}
   Hence $h_{u_i}^V=\dot{u}_{i,t}|_{t=0}+C$ according to the argument of Lemma \ref{HamiltonianGeodesic}. We have that
   \begin{align}
     d\tilde{\mathcal{F}}_\mu\cdot h_{u_i}^V=d\tilde{\mathcal{F}}_\mu\cdot \dot{u}_{i,t}|_{t=0}=0,
   \end{align}
   since $u_i$ is a minimum of $\tilde{\mathcal{F}}_\mu$.
\end{proof}

\begin{zj}
   After the modification, $u_i$ is still a critical point of $\mathcal{M}$, since $\mathcal{M}$ is invariant along the flow generated by Hamiltonian transverse holomorphic vector field.
\end{zj}

\medskip

Now, we give a proof of the main theorem.

\medskip

{\bf A proof of Theorem 1.1. } Let $g_{0}$ and $g_{1}$ are two cscS metrics compatible with $(M,\xi,\eta,\Phi,g)$, and $u_{0}$ and $u_{1}$ are the transverse K\"ahler potential  functions with respect to $g_{0}$ and $g_{1}$. By Proposition 4.5., modulo automorphisms, we can suppose that
 $u_i(i=0,1)$ satisfies that there exists $v_{i}$ such that
\begin{equation}
  D_{v_i}F|_{u_i}=-G_{\nu_i}.
\end{equation}
And now we   prove that $d\eta_{u_0}=d\eta_{u_1}$.

  Consider the functional $\mathcal{M}_s=\mathcal{M}+s\tilde{\mathcal{F}}_\mu$. Its differential is
  \begin{equation}
    F_s(u_i)=F(u_i)+sd\tilde{\mathcal{F}}_\mu(u_i)=F(u_i)+sG_{\nu_i}.
  \end{equation}
  For all $w_s$ considered as a tangent vector field in $\mathcal{H}$ along the curve $u_i+sv_i$ we have that
  \begin{equation}
  \label{00000}
    \left.\frac{d}{ds}\right|_{s=0} F_s(u_i+sv_i)\cdot w_s=D_{v_i}F|_{u_i}\cdot w_0 + F(u_i) D_{v_i}w_s +G_{\nu_i}\cdot w_0 =0.
  \end{equation}

  Since $F_s(u_i+sv_i)$ is a linear functional on the the space of basic functions, we have that $|F_s(u_i+sv_i)\cdot w|\leq C\mathop{\sup}\limits_M|w|o(s)$. Indeed, we can write
  \begin{equation}
    F_s(u_i+sv_i)\cdot w=\int_M w f(s,x) dV,
  \end{equation}
  for some function $f$ on $\mathbb{R}\times M$. Equation \eqref{00000} implies that $\left.\frac{\partial f(s,x)}{\partial s}\right|_{s=0}=0$. Furthermore $F_s(u_i+sv_i)\cdot w=0$ at $s=0$, arguing by the differential mean value theorem, we have that
  \begin{equation}
    \left| F_s(u_i+sv_i)\cdot \frac{w}{\mathop{\sup}\limits_M |w|}-0 \right| =|s|\left|\int_M \frac{w}{\mathop{\sup}\limits_M |w|} \left.\frac{\partial f(t,x)}{\partial t}\right|_{t=\theta(s,x)} dV\right|,
  \end{equation}
  where $|\theta(s,x)|<|s|$. Letting $s\to 0$, we conclude that $$|F_s(u_i+sv_i)\cdot w|\leq C\mathop{\sup}\limits_M|w|o(s).$$

  Lemma \ref{geodesic} implies that we can connect $u_0^s=u_0+sv_0$ and $u_1^s=u_1+sv_1$ by a unique weak $C^2$ geodesic $u_t^s$. In particular, according to the regularity in \cite{guan2012regularity}, we know that $||u_t^s||_{C^1(M\times [0,1])}$ and $|\triangle u_t^s|_{M\times [0,1]}$ is bounded for $s$ sufficiently small. According to Lemma \ref{diff mabuchi increasing}, the convexity of $\mathcal{M}$ implies that
  \begin{equation}
    \left.\frac{d}{dt}\right|_{t=0^+}\mathcal{M}(u_t^s)\geq F(u_0^s)\cdot \left. \frac{d u_t^s}{t} \right|_{t=0^+}
  \end{equation}
  and
  \begin{equation}
    \left.\frac{d}{dt}\right|_{t=1^-}\mathcal{M}(u_t^s)\leq F(u_1^s)\cdot \left. \frac{d u_t^s}{t} \right|_{t=1^-}.
  \end{equation}
  Since $\tilde{\mathcal{F}}_\mu$ is strictly convex along such weak geodesic, the same inequalities hold for $\mathcal{M}_s$ as well\footnote{$F$ on the righthand side of the above inequality will be replaced by $F_s$ as well.}. The linearity of $\mathcal{E}(u_t^s)$ in $t$ implies that
  \begin{equation}
    \begin{split}
      0&\leq s(\left.\frac{dI_\mu(u_t^s)}{dt}\right|_{t=1^-}-\left. \frac{dI_\mu(u_t^s)}{dt}\right|_{t=0^+})\\
      &\leq \left.\frac{d\mathcal{M}_s(u_t^s)}{dt}\right|_{t=1^-}- \left. \frac{d \mathcal{M}_s }{ dt} \right|_{ t=0^+}\\
      &\leq F_s(u_1^s)\cdot \left.\frac{du_t^s}{dt}\right|_{t=1^-}-F_s(u_0^s)\cdot \left.\frac{du_t^s}{dt}\right|_{t=0^+}\\
      &=o(s).
    \end{split}
  \end{equation}
  Hence,
  \begin{equation}
    \left.\frac{dI_\mu(u_t^s)}{dt}\right|_{t=1^-}-\left.\frac{dI_\mu(u_t^s)}{dt}\right|_{t=0^+} \leq o(1).
  \end{equation}
  According to Proposition \ref{geodist}, we know that $d(d\eta_{u^s_0},d\eta_{u^s_1})\leq o(1)$. Hence
  \begin{equation}
  d(d\eta_{u_0},d\eta_{u_1})=0,
  \end{equation}
  which implies that $d\eta_{u_0}=d\eta_{u_1}$. This complete the proof of main theorem.

\hfill $\Box$

%===================================================

%\bibliography{bib}
%\bibliographystyle{plain}

%===================================================

\end{document}